\documentclass[10pt,article,reqno]{amsart}
\usepackage{amssymb}
\usepackage{amsthm}
\usepackage{amsmath}
\usepackage{graphicx}
\usepackage{tikz}
\usepackage{tikz-3dplot}

\theoremstyle{plain}
\newtheorem{theorem}{Theorem}[section]
\newtheorem{lem}{Lemma}[section]

\numberwithin{equation}{section}

\newcommand\card{\#}
\newcommand\conv{\operatorname{conv}}

\def\R{\mathbb{R}}
\def\N{\mathbb{N}}

\def\T{\mathbb{T}}

\def\irA{\mathcal{A}}
\def\irB{\mathcal{B}}

\def\irI{\mathcal{I}}

\def\irM{\mathcal{M}}

\begin{document}

\title{Surjective Kuiper isometries}

\author{Gy\"orgy P\'al Geh\'er}
\address{MTA-SZTE Analysis and Stochastics Research Group \\
Bolyai Institute, University of Szeged \\
H-6720 Szeged, Aradi v\'ertan\'uk tere 1., Hungary}
\address{MTA-DE "Lend\"ulet" Functional Analysis Research Group, Institute of Mathematics\\
University of Debrecen\\
H-4010 Debrecen, P.O. Box 12, Hungary}
\email{gehergy@math.u-szeged.hu or gehergyuri@gmail.com}
\urladdr{http://www.math.u-szeged.hu/$\sim$gehergy/}


\begin{abstract}
Characterisations of surjective isometries with respect to the Kuiper distance on three classes of Borel probability measures of $\R$ (or equivalently, probability distribution functions) are presented here.
These classes are the set of continuous, absolute continuous and general measures.
\end{abstract}

\maketitle

\section{Intorduction and statements of the results}
The famous Banach--Stone theorem characterises surjective linear isometries between Banach spaces of continuous complex-valued functions on compact Hausdorff spaces equipped with the supremum norm.
Motivated by this theorem, G. Dolinar and L. Moln\'ar described surjective isometries of probability distribution functions with respect to the Kolmogorov--Smirnov distance in their joint paper \cite{KSisom1}.
It is important to note that the space of probability distribution functions is not a linear space, therefore there is no point in considering linearity of these maps.
Though, this space is a convex subset of a linear space, i.e. the space of all real-valued measures on $\R$, therefore invariance of convex combinations could be considered. 
But on the contrary, in \cite{KSisom1} the authors considered general surjective isometries, and the invariance of convex combinations under these transformations was not an assumption but a conclusion.

Later, L. Moln\'ar continued this project in three further publications \cite{KSisom2,Lisom,KSisom3}.
The present paper is a contribution to this line of research.
Namely, we will consider the so-called Kuiper distance and describe the structure of general surjective isometries on three classes of probability distribution functions.
However, unlike in the before mentioned papers, we will prefer the language of Borel probability measures instead of distribution functions.
We point out that as in the papers \cite{KSisom1,KSisom2,Lisom,KSisom3}, the invariance of convex combinations under these transformations will be a conclusion.

The space of all Borel probability measures will be denoted by $P(\R)$.
We call a $\mu\in P(\R)$ continuous if we have $\mu(\{x\}) = 0$ for every $x\in\R$, and absolute continuous if it is absolutely continuous with respect to the Gaussian measure (or equivalently with respect to the usual length measure $m$ on $\R$, though, we have to keep in mind that $m\notin P(\R)$).
The class of all continuous and absolutely continuous Borel probability measures on $\R$ will be denoted by $P_c(\R)$ and $P_{ac}(\R)$, respectively.
The symbol $\irB_\R$ will stand for the set of all Borel subsets of $\R$.

The distribution function of $\mu\in P(\R)$ is usually defined by $f_\mu(t) := \mu((-\infty,t])$ $(t\in\R)$.
It is well-known that $f_\mu$ is monotone increasing, continuous from the right, and satisfies $\lim_{t\to\infty} f_\mu(t) = 1$ and $\lim_{t\to -\infty} f_\mu(t) = 0$.
The Kolmogorov--Smirnov distance on $P(\R)$ is given by
$$
d_{KS}(\mu,\nu) = \sup_{t\in\R} |f_\mu(t)-f_\nu(t)| = \sup_{t\in\R} \big|\mu((-\infty,t])-\nu((-\infty,t])\big|,
$$
and the so-called total variation distance or statistical distance is 
$$
d_{TV}(\mu,\nu) = \sup \left\{ \big|\mu(B)-\nu(B)\big| \colon B\in\irB_\R \right\}.
$$

Let $\irI$ denote the set of all non-degenerate intervals of $\R$, i.e. 
$$
\irI = \{I\subseteq\R \colon \card I > 1 \text{ and } I \text{ is connected}\}.
$$
The set of all (possibly degenerate) intervals of $\R$ will be denoted by $\irI_0$, i.e. 
$$
\irI_0 = \irI \cup \big\{\{x\}\colon x\in\R\big\} = \{I\subseteq\R \colon I \text{ is connected and } I\neq\emptyset\}.
$$
The Kuiper metric is given by the following formula, where the second and third equations are easy to see:
\begin{equation}\label{Kuiper-def_eq}
\begin{gathered}
d_{Ku}(\mu,\nu) := \sup_{t\in\R} (f_\mu(t)-f_\nu(t)) + \sup_{t\in\R} (f_\nu(t)-f_\mu(t)) \\
= \sup\{|\mu(I)-\nu(I)|\colon I\in\irI\} = \sup\{|\mu(I)-\nu(I)|\colon I\in\irI,\, I\text{ is bounded}\}.
\end{gathered}
\end{equation}
This metric is a natural modified version of the above mentioned two distances. 
Obviously, we have $d_{Ku}(\mu,\nu) \leq 1$ $(\mu,\nu\in P(\R))$.
We will show in the next section that instead of supremum we can take maximum, if $I$ runs through $\irI_0$ instead of $\irI$.
This metric was defined by N.H. Kuiper in \cite{Kuiper}, and it seems that it is more useful in statistics than the Kolmogorov--Smirnov distance (\cite[pp. 39]{Da}).
It is also important because of the Kuiper density problem (\cite{DaKo}).

The goal of this paper is to present a characterisation of surjective isometries of $P_c(\R)$, $P_{ac}(\R)$ and $P(\R)$ with respect to the Kuiper distance.
We remark that this question was posed by Moln\'ar in a personal conversation.
Let $A\in\irB_\R$ and $g\colon A\to\R$ be an injective function which transforms Borel sets into Borel sets, i.e. $g(B) = g(B\cap A) \in \irB_\R$ $(B\in\irB_\R)$.
If $\mu\in P(\R)$, then by $\mu\circ g$ we mean the (not necessarily probability) Borel measure defined by 
$$
(\mu\circ g)(B) = \mu(g(B)) \quad (B\in\irB_\R).
$$
Clearly, in the special case when $g\colon A\to\R$ is surjective, we have $\mu\circ g \in P(\R)$.
We also point out that if $\mu\in P_c(\R)$ and $\card(\R\setminus g(A)) = 1$, then $\mu\circ g \in P_c(\R)$.

For every $x\in\R$ let us define the function 
$$
r_x\colon\R\setminus\{x\} \to \R, \quad r_x(t) = \frac{1}{t-x}; 
$$
and let $r_\infty\colon\R\to\R$ be the identity function.

Now, we formulate our main results on surjective Kuiper isometries.

\begin{theorem}\label{main-c_thm}
Let $\phi\colon P_c(\R) \to P_c(\R)$ be a surjective transformation which is an isometry with respect to $d_{Ku}$, i.e. we have
$$
d_{Ku}(\mu,\nu) = d_{Ku}(\phi(\mu),\phi(\nu)) \qquad (\mu,\nu\in P_c(\R)).
$$
Then there exists a homeomorphism $g\colon \R\to\R$ and an $x\in\R\cup\{\infty\}$ such that $\phi$ has the following form:
\begin{equation}\label{phiformc}
\phi(\mu) = \mu\circ (g \circ r_x) \qquad (\mu \in P_c(\R)).
\end{equation}
Moreover, transformations of the above form are all surjective isometries with respect to the Kuiper metric.
\end{theorem}

\begin{theorem}\label{main-ac_thm}
Let $\phi\colon P_{ac}(\R) \to P_{ac}(\R)$ be a surjective map which is an isometry with respect to the Kuiper metric, i.e. we have
$$
d_{Ku}(\mu,\nu) = d_{Ku}(\phi(\mu),\phi(\nu)) \qquad (\mu,\nu\in P_{ac}(\R)).
$$
Then there exists a bijective function $g\colon \R\to\R$ such that $g$ and $g^{-1}$ are locally (i.e. on every compact interval) absolutely continuous, and an $x\in\R\cup\{\infty\}$ such that we have
$$
\phi(\mu) = \mu\circ (g \circ r_x) \qquad (\mu \in P_{ac}(\R)).
$$
Moreover, every transformation with this form is a surjective Kuiper isometry on $P_{ac}(\R)$.
\end{theorem}

As a consequence of the above two theorems we have that every surjective Kuiper isometry on $P_{ac}(\R)$ can be extended to a surjective Kuiper isometry on $P_c(\R)$. 
But on the contrary, as we shall see from the next theorem, a similar conclusion does not hold for the classes $P(\R)$ and $P_c(\R)$.
The reason is the following: unlike on $P_c(\R)$, all surjective Kuiper isometries on $P(\R)$ transform measures with compact support to measures of the same type.

\begin{theorem}\label{main-gen_thm}
Let $\phi\colon P(\R) \to P(\R)$ be a surjective Kuiper isometry, i.e. we have
$$
d_{Ku}(\mu,\nu) = d_{Ku}(\phi(\mu),\phi(\nu)) \qquad (\mu,\nu\in P(\R)).
$$
Then there exists a homeomorphism $g\colon \R\to\R$ such that
$$
\phi(\mu) = \mu\circ g \qquad (\mu \in P(\R)).
$$
Moreover, every transformation of this form is a surjective Kuiper isometry on $P(\R)$.
\end{theorem}

The first two theorems will be proven in Section \ref{c-ac_thm_sec}.
Our method will be to transform our problem to another one which considers Borel probability measures on the unit circle $\T$ of $\R^2$.
This will help us to avoid some technical difficulties.
The last result will be verified in Section \ref{gen_thm_sec}.
In order to give that characterisation, our first step will be to establish a metric characterisation of Dirac measures in a similar way as it was done in \cite{KSisom1} and \cite{Lisom}.
After that we shall utilise Theorem \ref{main-c_thm}.


\section{Proofs of the continuous cases} \label{c-ac_thm_sec}

We begin with proving that in \eqref{Kuiper-def_eq} we can take maximum if $I$ runs through $\irI_0$.

\begin{lem}\label{Kuiper-c-max_lem}
For every $\mu,\nu\in P(\R)$ we have the following equation:
\begin{equation}\label{Kuiper_eq}
d_{Ku}(\mu,\nu) = \max\{|\mu(I)-\nu(I)|\colon I\in\irI_0\}.
\end{equation}
\end{lem}

\begin{proof}
If $d_{Ku}(\mu,\nu) = 0$, then we may choose $I = \{x_0\}$ with $\mu(\{x_0\}) = \nu(\{x_0\}) = 0$.
Therefore we may assume throughout the proof that $d_{Ku}(\mu,\nu) > 0$.

First, let us assume that $f_\mu \leq f_\nu$ (the $f_\mu \geq f_\nu$ case is similar).
Then we clearly have 
$$
d_{Ku}(\mu,\nu) = \sup_{t\in\R} (f_\nu(t)-f_\mu(t)).
$$
Since the limits of the distribution functions at $-\infty$ and $\infty$ are finite and they coincide, we can find a bounded sequence $\{t_n\}_{n=1}^\infty$ such that $\lim_{n\to\infty} f_\nu(t_n)-f_\mu(t_n) = d_{Ku}(\mu,\nu)$.
Since every bounded sequence has a convergent subsequence, we may assume that our sequence was convergent.
Let $t_0 := \lim_{n\to\infty} t_n$.
We can also suppose that we have either $t_n \geq t_0$ $(n\in\N)$ or $t_n < t_0$ $(n\in\N)$.
If the first possibility occurs, then we obtain
$$
d_{Ku}(\mu,\nu) = \lim_{n\to\infty} f_\nu(t_n)-f_\mu(t_n) = f_\nu(t_0)-f_\mu(t_0) = \nu((-\infty,t_0])-\mu((-\infty,t_0]).
$$
For the second one, we infer
$$
d_{Ku}(\mu,\nu) = \lim_{n\to\infty} f_\nu(t_n)-f_\mu(t_n) = f_\nu(t_0-)-f_\mu(t_0-) = \nu((-\infty,t_0))-\mu((-\infty,t_0)).
$$

Second, we assume that neither $f_\mu \leq f_\nu$ nor $f_\mu \geq f_\nu$ holds.
In this case there exist bounded sequences $\{t_n\}_{n=1}^\infty$ and $\{s_n\}_{n=1}^\infty$ such that
$$
d_{Ku}(\mu,\nu) = \lim_{n\to\infty} (f_\mu(t_n)-f_\nu(t_n)) + \lim_{n\to\infty} (f_\nu(s_n)-f_\mu(s_n)).
$$
Like in the previous case, we may assume that these sequences are convergent. 
Let $s_0 := \lim_{n\to\infty} s_n$ and $t_0 := \lim_{n\to\infty} t_n$.
We can also suppose that we have either $t_n \geq t_0$ $(n\in\N)$ or $t_n < t_0$ $(n\in\N)$; and a similar assumption can be made on $\{s_n\}_{n=1}^\infty$.
Let us assume that we have $s_n < s_0$ and $t_0 \leq t_n$ $(n\in\N)$.
Then we deduce the following:
$$
d_{Ku}(\mu,\nu) = f_\mu(t_0)-f_\nu(t_0) + f_\nu(s_0-)-f_\mu(s_0-) 
= \left\{
\begin{matrix}
\mu([s_0,t_0]) - \nu([s_0,t_0]) & \text{if } s_0 \leq t_0\\
\nu((t_0,s_0)) - \mu((t_0,s_0)) & \text{if } t_0 < s_0\\
\end{matrix}
\right..
$$
The other cases can be handled similarly.
\end{proof}

Let $\xi_t = (\cos t, \sin t)\in\R^2$ $(t\in\R)$, and let us consider the unit circle $\T = \{\xi_t \colon t\in[-\pi,\pi)\}$, and the following continuous map:
$$
\tau\colon \R\to\T, \; \tau(t) = \xi_{2 \arctan t}.
$$
The set of all (possibly degenerate) arcs of $\T$ is denoted by $\irA_0$.
The sets of all non-degenerate arcs, closed (possibly degenerate) arcs, and open arcs will be denoted by the symbols $\irA$, $\irA^{cl}$ and $\irA^{op}$, respectively.
(Note that none of $\irA$, $\irA_0$, $\irA^{cl}$ or $\irA^{op}$ contains the empty set).
The spaces of all continuous and absolutely continuous (with respect to the normalised arc-length measure $\lambda$ on $\T$) Borel probability measures will be denoted by $P_c(\T)$ and $P_{ac}(\T)$, respectively.
We define the following function:
\begin{equation}\label{T-Kuiper-def_eq}
d\colon P_c(\T)\times P_c(\T) \to \R, \; d(\mu,\nu) = \max\{|\mu(A)-\nu(A)|\colon A\in\irA_0\}.
\end{equation}
Clearly, the map $\tau$ is a homeomorphism between $\R$ and $\T\setminus\{\xi_\pi\}$.
We observe the following:
$$
d(\mu,\nu) = d_{Ku}(\mu\circ\tau,\nu\circ\tau) \quad (\mu,\nu\in P_c(\T)),
$$
where $\mu\circ\tau \in P_c(\R)$ is defined by $(\mu\circ\tau)(B) = \mu(\tau(B))$ $(B\in\irB_\R)$.
In order to verify this, we observe two properties.
First, that we have $d_{Ku}(\mu,\nu) = \max\{|\mu(\R\setminus I)-\nu(\R\setminus I)|\colon I\in\irI_0\}$.
Second, that for every $I \in\irI_0$ the sets $\tau(I)$ and $\{-1\}\cup\tau(\R\setminus I)$ are arcs of $\T$, and that for every arc $A\subset\T$ the set $\tau^{-1}(A)$ (inverse image) is either an interval or the complement of an interval.
Therefore we conclude that $d$ is a metric on $P_c(\T)$. 
We note that $\mu$ is absolutely continuous exactly when $\mu\circ\tau$ is.
In fact, this is a consequence of the fact that the derivatives of the tangent and arctangent functions are bounded from below and above by some positive numbers on every compact interval of $(-\tfrac{\pi}{2},\tfrac{\pi}{2})$ and $\R$, respectively.

In order to prove Theorems \ref{main-c_thm} and \ref{main-ac_thm} first, we investigate surjective isometries on $P_c(\T)$ and $P_{ac}(\T)$ with respect to the metric $d$.
For any $\mu\in P_c(\T)$ we will use the following notation:
$$
\{\mu\}^c := \{\nu\in P_c(\T) \colon d(\mu,\nu) = 1\}.
$$
If $\mu\in P_{ac}(\T)$, then we define 
$$
\{\mu\}^{ac} := \{\nu\in P_{ac}(\T) \colon d(\mu,\nu) = 1\}.
$$
The closed support of $\mu\in P_c(\T)$ (or $\mu\in P(\R)$, respectively) is the smallest closed set $S_\mu$ of $\T$ (or $\R$, resp.) such that we have $\mu(S_\mu) = 1$, or equivalently, the complement of the union of all open arcs (or intervals, resp.) which have zero $\mu$-measure.
It is an elementary fact that every non-empty (relatively) open subset $U$ of $\T$ (or $\R$, resp.) can be written as a countable disjoint union of open arcs (or intervals, resp.): $U = \cup_{j=1}^n U_j$, where $n\in\N\cup\{\infty\}$ and $U_j$'s are exactly the connected components of $U$.
Now, we give a characterisation of the set $\{\mu\}^c$ for every $\mu\in P_c(\T)$.

\begin{lem}\label{mu-c-char_lem}
Let $\mu\in P_c(\T)$ and $\T\setminus S_\mu = \cup_{j=1}^n U_j$ where $n\in\N\cup\{\infty\}$ and $U_j$'s are the connected components of $\T\setminus S_\mu$.
Then we have 
\begin{equation}\label{mu-c_eq}
\{\mu\}^c = \bigcup_{j=1}^n \{\nu\in P_c(\T) \colon \nu(U_j) = 1\}.
\end{equation}
In case when $\mu\in P_{ac}(\T)$, then we have
\begin{equation}\label{mu-ac_eq}
\{\mu\}^{ac} = \bigcup_{j=1}^n \{\nu\in P_{ac}(\T) \colon \nu(U_j) = 1\}.
\end{equation}
Moreover, the terms in the unions of \eqref{mu-c_eq} and \eqref{mu-ac_eq} are exactly the connected components, i.e. the sets $\{\mu\}^c$ and $\{\mu\}^{ac}$ have exactly $n$ connected components.
In particular, $\{\mu\}^c$ (or $\{\mu\}^{ac}$) is non-empty and connected if and only if $S_\mu \neq \T$ and $S_\mu$ is connected; 
and $\{\mu\}^c = \emptyset$ (or $\{\mu\}^{ac} = \emptyset$) exactly when $S_\mu = \T$ holds.
\end{lem}

\begin{proof}
We have
$$
\begin{gathered}
\{\mu\}^c = \{\nu\in P_c(\T) \colon \exists\; A \in \irA_0 \text{ such that } \\ \text{either } \mu(A) = 0 \text{ and } \nu(A) = 1, \text{ or } \mu(\T\setminus A) = 0  \text{ and } \nu(\T\setminus A) = 1\}.
\end{gathered}
$$
Since we have $\T\setminus A \in \irA_0$ for every $A \in \irA_0$, and the measure of a closed arc and its interior are the same for continuous measures, we infer
$$
\{\mu\}^c = \{\nu\in P_c(\T) \colon \exists\; A\in\irA^{op} \text{ such that } \mu(A) = 0, \nu(A) = 1\}.
$$
But $\mu(A) = 0$ holds if and only if $A\cap S_\mu = \emptyset$, therefore we conclude \eqref{mu-c_eq}.
The equation \eqref{mu-ac_eq} can be obtained in a similar way.

For the other statement we make two observations.
On one hand, if $\nu_1,\nu_2\in\{\mu\}^c$, $\nu_1(U_{j_1}) = \nu_2(U_{j_2}) = 1$ and $j_1\neq j_2$, then we easily conclude
$$
1\geq d(\nu_1,\nu_2) \geq |\nu_1(U_{j_1})-\nu_2(U_{j_1})| = 1,
$$
and thus $d(\nu_1,\nu_2) = 1$.
This readily implies that the number of connected components of $\{\mu\}^c$ is at least $n$.
On the other hand, if $\nu_1,\nu_2\in\{\mu\}^c$, $\nu_1(U_{j_0}) = \nu_2(U_{j_0}) = 1$, then we consider
$$
\gamma\colon [0,1] \to P_c(\T), \; \gamma(t) = (1-t)\cdot\nu_1 + t\cdot\nu_2.
$$
Since we have 
$$
d(\gamma(s), \gamma(t)) = \max\{|(1-t)\cdot\nu_1(A) + t\cdot\nu_2(A)-(1-s)\cdot\nu_1(A) - s\cdot\nu_2(A)| \colon A\subset\T \text{ is an arc}\}
$$
$$
= |s-t|\cdot\max\{|\nu_1(A) - \nu_2(A)| \colon A\subset\T \text{ is an arc}\} = |s-t|\cdot d(\nu_1,\nu_2) \leq |s-t|,
$$
the curve $\gamma$ connects $\nu_1$ with $\nu_2$ in $\{\nu\in P_c(\T) \colon \nu(U_{j_0}) = 1\}$, whence the connectedness of $\{\nu\in P_c(\T) \colon \nu(U_{j_0}) = 1\}$ is yielded.
Finally, the proof for $\{\mu\}^{ac}$ is almost the same.
\end{proof}

Next, we set
$$
P_c^{cs}(\T) := \big\{\mu\in P_c(\T) \colon \{\mu\}^c = \emptyset \text{ or } \{\mu\}^c \text{ is connected} \big\} = \{\mu\in P_c(\T) \colon S_\mu \in\irA^{cl}\}
$$ 
and 
$$
P_{ac}^{cs}(\T) := \big\{\mu\in P_{ac}(\T) \colon \{\mu\}^{ac} = \emptyset \text{ or } \{\mu\}^{ac} \text{ is connected} \big\} = \{\mu\in P_{ac}(\T) \colon S_\mu \in\irA^{cl}\}.
$$
Let $\psi\colon P_c(\T)\to P_c(\T)$ be an arbitrary surjective isometry with respect to the metric $d$.
Since $\psi$ is also a homeomorphism, we infer that
$$
\{\mu\}^c = \emptyset \;\iff\; \psi\left(\{\mu\}^c\right) = \{\psi(\mu)\}^c = \emptyset
$$
and
$$
\{\mu\}^c \text{ is conncected} \;\iff\; \{\psi(\mu)\}^c \text{ is conncected,}
$$
whence 
$$
\psi(P_c^{cs}(\T)) = P_c^{cs}(\T)
$$
is yielded.
By \eqref{mu-c_eq} the following equivalences are straightforward:
\begin{equation}\label{bekebelez0}
\begin{gathered}
S_\mu \subseteq S_\nu \;\iff\; \{\nu\}^c \subseteq \{\mu\}^c \;\iff\; \{\psi(\nu)\}^c \subseteq \{\psi(\mu)\}^c \\
\;\iff\; S_{\psi(\mu)} \subseteq S_{\psi(\nu)} \quad (\mu,\nu \in P_c(\T))
\end{gathered}
\end{equation}
and 
\begin{equation}\label{egyenlo0}
\begin{gathered}
S_\mu = S_\nu \;\iff\; \{\nu\}^c = \{\mu\}^c \;\iff\; \{\psi(\nu)\}^c = \{\psi(\mu)\}^c \\
\;\iff\; S_{\psi(\mu)} = S_{\psi(\nu)} \quad (\mu,\nu \in P_c(\T)).
\end{gathered}
\end{equation}
Because of the above observations the following map can be defined:
$$
\eta_\psi\colon \irA^{cl}\to\irA^{cl}, \quad \eta_\psi(S_\mu) = S_{\psi(\mu)} \qquad (\mu\in P_c^{cs}(\T)).
$$
Since $\psi^{-1}$ is also a surjective isometry, we obtain $\eta_\psi^{-1} = \eta_{\psi^{-1}}$.
Clearly, we have the following property:
\begin{equation}\label{bekebelez}
L\subseteq K \;\iff\; \eta_\psi(L)\subseteq\eta_\psi(K) \;\iff\; \eta_\psi^{-1}(L)\subseteq\eta_\psi^{-1}(K) \quad (L, K \in \irA^{cl}).
\end{equation}
In fact, more is true, which is proven in the next lemma.
Let us point out that the verification of the analogues of the previous observations for $P_{ac}(\T)$ is very similar.
We will denote by $\lambda$ the normalised arc-length measure on $\T$.

\begin{lem}\label{homeomorphism_lem}
Let $\psi\colon P_c(\T)\to P_c(\T)$ be a surjective isometry with respect to the metric $d$.
Then there exists a homeomorphism $h\colon \T\to\T$ such that we have
\begin{equation}\label{hKc}
S_{\psi(\mu)} = \eta_\psi(S_\mu) = h(S_{\mu}) \qquad (\mu \in P_c^{cs}(\T)).
\end{equation}

Moreover, if $\psi\colon P_{ac}(\T)\to P_{ac}(\T)$ is a surjective isometry with respect to the metric $d$, then 
\begin{equation}\label{hKac}
S_{\psi(\mu)} = \eta_\psi(S_\mu) = h(S_{\mu}) \qquad (\mu \in P_{ac}^{cs}(\T)).
\end{equation}
holds with a homeomorphism $h\colon \T\to\T$ which preserves sets with zero Lebesgue measure in both directions, i.e. we have 
\begin{equation}\label{LuzinNforh}
\lambda(A) = 0 \;\iff\; \lambda(h(A)) = 0 \quad (A\in\irB_\T).
\end{equation}
\end{lem}

\begin{proof}
We begin with the first statement.
Let us define a mapping $h\colon \T\to\T$ such that for each $t\in[-\pi,\pi)$ the point $h(\xi_t)$ is an arbitrary one lying in the intersection 
\begin{equation}\label{metszet}
\bigcap_{j=1}^\infty \eta_\psi\big(\{\xi_s\colon |s-t|\leq 1/j\}\big).
\end{equation}
By \eqref{bekebelez} and completeness of the metric of $\T$, the above intersection is a non-empty closed arc, thus $h$ is indeed a function.
Let us assume for a moment that there is a $t\in [-\pi,\pi)$ such that the intersection in \eqref{metszet} is not a point, but a non-degenerate closed arc $K\in\irA^{cl}$.
In this case, \eqref{bekebelez} readily implies the contradiction $\card(\eta_\psi^{-1} (K))\leq 1$.
Therefore each of the above intersections has exactly one element, hence the map $h$ is uniquely determined by the above properties.

Next, we observe that for every $t\in[-\pi,\pi)$ and $\{\alpha_k\}_{k=1}^\infty, \{\beta_k\}_{k=1}^\infty \subset [0,\infty)$ with $\alpha_k \searrow 0$, $\beta_k \searrow 0$ $(k\to\infty)$, $\alpha_k + \beta_k > 0$ $(k\in\N)$ we have
\begin{equation}\label{metszet2}
\bigcap_{k=1}^\infty \eta_\psi\big(\{\xi_s\colon s\in[t-\alpha_k,t+\beta_k] \}\big) = \{h(\xi_t)\}.
\end{equation}
In order to verify this, we observe that there exists a non-decreasing sequence of positive integers $\{j_k\}_{k=1}^\infty$ such that $\lim_{k\to\infty} j_k = \infty$ and $[t-\alpha_k,t+\beta_k] \subseteq [t-1/j_k,t+1/j_k]$ $(k\in\N)$.
Thus, again by \eqref{bekebelez}, the intersection in \eqref{metszet2} is a subset of $\{h(\xi_t)\}$.
But clearly this intersection cannot be empty.

The above observations imply $h(K) \subseteq \eta_\psi(K)$ $(K \in \irA^{cl})$.
Let $K\in \irA^{cl}$ be arbitrary, and let us consider a point $\chi \in \eta_\psi(K)$.
There is a monotone decreasing sequence $\{K_j\}_{j=1}^\infty \subset \irA^{cl}$ such that $K_j\subset \eta_\psi(K)$ $(j\in\N)$ and $\cap_{j=1}^\infty K_j = \{\chi\}$.
Clearly, $\{\eta_{\psi^{-1}}(K_j)\}_{j=1}^\infty = \{\eta_\psi^{-1}(K_j)\}_{j=1}^\infty \subset \irA^{cl}$ is a monotone decreasing sequence of closed arcs such that $\eta_\psi^{-1}(K_j)\subset K$ and $\card\left(\cap_{j=1}^\infty \eta_\psi^{-1}(K_j)\right) = \card\left(\cap_{j=1}^\infty \eta_{\psi^{-1}}(K_j)\right) = 1$ $(j\in\N)$.
Therefore $h(K) = \eta_\psi(K)$, which verifies \eqref{hKc}.

It remains to show that $h$ is a homeomorphism.
Let $\mu, \nu \in P_c^{cs}(\T)$ be arbitrary and set $K = S_\mu$, $L = S_\nu$.
We have 
$$
\{\mu\}^c\cap\{\nu\}^c\cap P_c^{cs}(\T) = \{\vartheta\in P_c^{cs}(\T)\colon \vartheta(\T\setminus(K\cup L)) = 1\}.
$$
On one hand, if $K \cap L = \emptyset$, then $\T\setminus(K\cup L)$ has exactly two connected components $U_1$ and $U_2$, which implies that in this case $\{\mu\}^c\cap\{\nu\}^c\cap P_c^{cs}(\T) = \cup_{j=1}^2 \{\vartheta\in P_c^{cs}(\T)\colon \vartheta(U_j) = 1\}$.
But for any choices $\vartheta_j \in \{\vartheta\in P_c^{cs}(\T)\colon \vartheta(U_j) = 1\}$ $(j=1,2)$, we clearly have $d(\vartheta_1,\vartheta_2) = 1$, which implies that $\{\mu\}^c\cap\{\nu\}^c\cap P_c^{cs}(\T)$ is not a connected set in this case.
On the other hand, if $K \cap L \neq \emptyset$, then $\T\setminus(K\cup L)$ is a connected open set, moreover, we claim that $\{\mu\}^c\cap\{\nu\}^c\cap P_c^{cs}(\T)$ is connected. 
Let $\vartheta_1, \vartheta_2 \in \{\mu\}^c\cap\{\nu\}^c\cap P_c^{cs}(\T)$.
We choose a $\widetilde\vartheta \in P_c^{cs}(\T)$ with 
$$
S_{\vartheta_1} \cup S_{\vartheta_2} \subseteq S_{\widetilde\vartheta} \subseteq \big(\T\setminus(K\cup L)\big)^-,
$$ 
where $\cdot^-$ means the closure of a given set, and consider the following path
$$
\gamma\colon [0,1] \to P_c(\T), \; \gamma(t) = (t-t^2)\cdot\widetilde\vartheta + (1-t+t^2)\cdot\big[ (1-t)\cdot\vartheta_1 + t\cdot\vartheta_2 \big].
$$
Clearly, $\gamma([0,1]) \subseteq \{\mu\}^c\cap\{\nu\}^c\cap P_c^{cs}(\T)$, $\gamma(0) = \vartheta_1$, $\gamma(1) = \vartheta_2$, and a straightforward computation gives the continuity of $\gamma$. 
Thus, indeed, the set $\{\mu\}^c\cap\{\nu\}^c\cap P_c^{cs}(\T)$ is connected.

Applying these observations we get the following:
\begin{equation}\label{hmetszet}
K \cap L = \emptyset \; \iff \; h(K) \cap h(L) = \emptyset \quad (K,L \in\irA^{cl}).
\end{equation}

Clearly, $h(\T) = \eta_\psi(\T) = \T$, thus $h$ is surjective.
Let us consider two different points $\xi_s$ and $\xi_t$ $(s\neq t, s,t\in[-\pi,\pi))$.
Then there are two disjoint closed arcs $K_s, K_t \in \irA^{cl}$ such that $\xi_s \in K_s$ and $\xi_t \in K_t$.
By \eqref{hmetszet} we have $h(K_s) \cap h(K_t) = \emptyset$, whence we infer the bijectivity of $h$.
Finally, let $\{t_n\}_{n=1}^\infty \subset \R$ be a sequence such that $t_n \searrow t$ or $t_n \nearrow t$ $(n\to\infty)$ and $|t_1-t| < \tfrac{\pi}{2}$, and let $K_n$ be the shorter closed arc with endpoints $\xi_t$ and $\xi_{t_n}$.
Since we have $\{h(\xi_t)\} = \cap_{n=1}^\infty h(K_n)$ and $h(\xi_{t_n}) \in h(K_n)$ $(n\in\N)$, the continuity of $h$ is yielded.
But $h$ is a bijective continuous map of the compact Hausdorff space $\T$, thus we conclude that $h$ is a homeomorphism.
This completes the proof for the first case.

For the second case, we obtain in a similar way as above that there is a homeomorphism $h\colon\T\to\T$ such that \eqref{hKac} is satisfied.
Let $K\subseteq\T$ be a compact set with $\lambda(K)>0$, and $\T\setminus K = \cup_{j=1}^n U_j$ ($n\in\N\cup\{\infty\}$) where the union is disjoint and $U_j\in \irA^{op}$ for every $j$.
We intend to show that $\lambda(h(K)) > 0$ is satisfied.
Obviously, there exists a $\mu\in P_{ac}(\T)$ with $S_\mu = K$.
By \eqref{bekebelez0} and \eqref{hKac} we infer
$$
h(K) = h\left(\bigcap_{j=1}^n (\T\setminus U_j)\right) = \bigcap_{j=1}^n h(\T\setminus U_j) = \bigcap_{j=1}^n \eta_\psi(\T\setminus U_j) \supseteq S_{\psi(\mu)}.
$$
Thus we obtain $0 < \lambda(S_{\psi(\mu)}) \leq \lambda(h(K))$.
Next, if $A\in\irB_\T$ has positive Lebesgue measure, then by regurality we infer that there is a compact subset $K\subseteq A$ such that $K$ has still positive Lebesgue measure.
Therefore we conclude 
$$
\lambda(h(A)) = 0 \;\Longrightarrow\; \lambda(A) = 0 \quad (A\in\irB_\T).
$$
For the reverse direction we only have to observe the following property which we can conclude from \eqref{hKac}:
$$
S_{\psi^{-1}(\mu)} = \eta_{\psi^{-1}}(S_\mu) = \eta_{\psi}^{-1}(S_\mu) = h^{-1}(S_{\mu}) \qquad (\mu \in P_{ac}^{cs}(\T)).
$$
This completes the proof.
\end{proof}

We proceed with the verification of the next lemma where measures on $\R$ are considered.

\begin{lem}\label{inf_lem}
Let $I\in\irI$ be a closed (possibly unbounded) interval, and $\mu\in P_c(\R)$.
We have the following equation:
\begin{equation}\label{inf_eq}
1 - \mu(I) = \inf\{d_{Ku}(\mu,\vartheta) \colon \vartheta\in P_c(\R), S_\vartheta \subseteq I\}.
\end{equation}
Furthermore, if $\mu\in P_{ac}(\R)$, then we have 
\begin{equation}\label{infac_eq}
1 - \mu(I) = \inf\{d_{Ku}(\mu,\vartheta) \colon \vartheta\in P_{ac}(\R), S_\vartheta \subseteq I\}.
\end{equation}
\end{lem}

\begin{proof}
We will only deal with \eqref{inf_eq}, since \eqref{infac_eq} can be handled in a similar way.
We have 
$$
d_{Ku}(\mu,\vartheta) \geq |\mu(I)-\vartheta(I)| = 1 - \mu(I) \quad (\vartheta\in P_c(R),\, S_\vartheta \subseteq I),
$$
which implies $1 - \mu(I) \leq \inf\{d_{Ku}(\mu,\vartheta) \colon \vartheta\in P_c(R), S_\vartheta \subseteq I\}$.
If $\mu(I) = 0$, then we immediately obtain \eqref{inf_eq}, thus in the sequel we may assume that we have $\mu(I) > 0$.

Let us define $\nu \in P(\R)$ by
$$
\nu(A) = \frac{\mu(A\cap I)}{\mu(I)} \quad (A\in \irB_\R).
$$
Clearly, we have $\nu \in P_c(\R)$ and $S_\nu \subseteq I$, moreover 
\begin{equation}\label{Inelszukebb}
|\nu(J) - \mu(J)| = \nu(J) - \mu(J) = \frac{\mu(J)}{\mu(I)} (1 - \mu(I)) \leq 1 - \mu(I) \quad (J\in\irI,\, J\subseteq I).
\end{equation}
Now, we consider an arbitrary interval $\widetilde I\in\irI$.
There are two possibilities: either $\mu(\widetilde I) -\nu(\widetilde I) \leq 0$, or $\mu(\widetilde I) -\nu(\widetilde I) > 0$.
In case of the first one, we have the following estimation:
$$
\left|\mu(\widetilde I) -\nu(\widetilde I)\right| = \nu(\widetilde I) -\mu(\widetilde I) \leq \nu(I \cap \widetilde I) -\mu(I \cap \widetilde I) \leq 1-\mu(I),
$$
where we used \eqref{Inelszukebb}. 
For the second possibility, $\widetilde I \subseteq I$ is impossible, thus, again by \eqref{Inelszukebb}, we estimate in the following way:
$$
\left|\mu(\widetilde I) -\nu(\widetilde I)\right| 
= \mu(\widetilde I) -\nu(\widetilde I) 
= \nu(\R \setminus \widetilde I) -\mu(\R \setminus \widetilde I) 
= \nu\left((\R \setminus \widetilde I)\cap I\right) -\mu(\R \setminus \widetilde I) 
$$
$$
\leq \nu\left((\R \setminus \widetilde I)\cap I\right) -\mu\left((\R \setminus \widetilde I)\cap I\right) 
\leq 1-\mu(I),
$$
where we observed that $(\R \setminus \widetilde I)\cap I$ is an interval.
Since $1-\mu(I) = \nu(I) - \mu(I)$, we conclude $d_{Ku}(\mu,\nu) = 1-\mu(I)$.
Therefore \eqref{inf_eq} follows, which ends the proof.
\end{proof}

For an arbitrary $\mu\in P_c(\R)$ we define the measure $\mu\circ\tau^{-1}$ by $(\mu\circ\tau^{-1})(A) = \mu(\tau^{-1}(A))$ $(A\in\irB_\T)$, where $\tau^{-1}(A)$ is the inverse image of $A$.
It is straightforward that the transformation $\mu \mapsto \mu\circ\tau^{-1}$ is a bijection between $P_c(\R)$ and $P_c(\T)$.
Furthermore, $\mu$ is absolutely continuous if and only if $\mu\circ\tau^{-1} \in P_{ac}(\T)$.
We define in a very similar way the measure $\widetilde\mu\circ h \in P_c(\T)$ where $\widetilde\mu \in P_c(\T)$ and $h\colon\T\to\T$ is a homeomorphism.
Let us note that in case when $\widetilde\mu \in P_{ac}(\T)$, then $\widetilde\mu\circ h$ is not necessarily absolutely continuous.

Now, we are in the position to prove our first theorem.

\begin{proof}[Proof of Theorem \ref{main-c_thm}]
Let 
$$
\psi\colon P_c(\T)\to P_c(\T), \quad \psi(\widetilde\mu) = \phi(\widetilde\mu\circ\tau)\circ\tau^{-1},
$$ 
It is clear that $\phi$ is a surjective isometry with respect to the Kuiper metric if and only if $\psi$ is a surjective isometry with respect to the metric $d$.

By Lemma \ref{homeomorphism_lem}, there exists a homeomorphism $h\colon\T\to\T$ such that 
$$
S_{\psi(\widetilde\mu)} = h(S_{\widetilde\mu}) \quad (\widetilde\mu\in P_c^{cs}(\T)).
$$
Let us consider the following mapping:
$$
\psi_1\colon P_c(\T)\to P_c(\T), \quad \psi_1(\widetilde\mu) = (\psi(\widetilde\mu))\circ h.
$$
It is straightforward to see that $\psi_1$ is a surjective isometry if and only if $\psi$ is, and that we have 
$$
S_{\psi_1(\widetilde\mu)} = h^{-1}(S_{\psi(\widetilde\mu)}) = S_{\widetilde\mu} \quad (\widetilde\mu\in P_c^{cs}(\T)).
$$
We define the transformation
$$
\phi_1\colon P_c(\R)\to P_c(\R), \quad \phi_1(\mu) = \psi_1(\mu\circ\tau^{-1})\circ\tau.
$$
It is apparent that $\phi_1$ is a surjective isometry if and only if $\phi$ is, and that we have 
\begin{equation}\label{phi1suppegyenlo}
S_{\phi_1(\mu)} = S_{\mu} \quad (\mu\in P_c^{cs}(\R)),
\end{equation}
where 
$$
P_c^{cs}(\R) = \{\mu\in P_c(\R) \colon S_\mu\in\irI \text{ or } \R\setminus S_\mu\in\irI\}.
$$
The following equivalence follows easily from the definition of $\phi_1$ and the property \eqref{bekebelez0} for $\psi_1$:
\begin{equation}\label{tartalmaz2}
S_\mu \subseteq S_\nu \;\iff\; S_{\phi_1(\mu)} \subseteq S_{\phi_1(\nu)} \quad (\mu,\nu\in P_c(\R)).
\end{equation}

Now, let $\mu$ be an arbitrary continuous Borel probability measure on $\R$.
Since we have $\psi_1(P_c^{cs}(\T)) = P_c^{cs}(\T)$, we infer $\phi_1(P_c^{cs}(\R)) = P_c^{cs}(\R)$.
This, \eqref{phi1suppegyenlo}, \eqref{tartalmaz2} and Lemma \ref{inf_lem} implies the following for every closed interval $I\in\irI$:
$$
\phi_1(\mu)(I) = 1 - \inf\{d_{Ku}(\phi_1(\mu),\vartheta) \colon \vartheta\in P_c(\R), S_\vartheta \subseteq I\} 
$$
$$
= 1 - \inf\{d_{Ku}(\mu,\phi_1^{-1}(\vartheta)) \colon \vartheta\in P_c(\R), S_\vartheta \subseteq I\} 
$$
$$
= 1 - \inf\{d_{Ku}(\mu,\vartheta) \colon \vartheta\in P_c(\R), S_\vartheta \subseteq I\} = \mu(I).
$$
But this immediately implies $\phi_1(\mu) = \mu$, and thus that $\phi_1$ is the identity map.
It is tedious, but straightforward, to check that transforming back to our original map $\phi$ yields \eqref{phiformc} where
$$
x = \left\{
\begin{matrix}
\infty & \text{if } h(-1) = -1\\
\tau^{-1}(h(-1)) & \text{otherwise}
\end{matrix}
\right.
$$
and $g$ is the continuous extension of $\tau^{-1}\circ h^{-1}\circ \tau \circ r_x^{-1}$. 
Note that the latter function is not defined in at most two points of $\R$ (depending on the actual value of $x$ and the function $h$). 
However, if we consider $r_x$ as a bijective function of the one-point compactification of $\R$ (which is topologically equivalent to $\T$), then it is not hard to see that $g$ (as the continuous extension of $\tau^{-1}\circ h^{-1}\circ \tau \circ r_x^{-1}$) makes sense and that it is indeed a homeomorphism of $\R$.
\end{proof}

The proof of our second result is quite similar to the above one, therefore we only present its sketch.

\begin{proof}[Proof of Theorem \ref{main-ac_thm}]
The definitions of $\psi$, $\psi_1$ and $\phi_1$ are similar to the above definitions.
It is straightforward that each of these transformations is a surjective Kuiper isometry of $P_{ac}(\T)$ if and only if $\phi$ is.
While transforming back the identity map (i.e. $\phi_1$) to $\phi$ we observe that the homeomorphism $g$ from the statement has the property
$$
m(B) = 0 \;\iff\; m(g(B)) = 0 \quad (B\in\irB_\R).
$$
Since every homeomorphism of $\R$ is either monotone increasing or decreasing, we infer that $g$ and $g^{-1}$ are of bounded variation on every compact interval.
Therefore the famous Banach--Zarecki theorem (see e.g. \cite{LuzinN1,LuzinN2}) implies that they are both locally absolutely continuous functions.
This completes one direction of the statement.
The other direction is a rather easy calculation.
\end{proof}


\section{Proof in the general case} \label{gen_thm_sec}

In order to prove our last theorem, we need to introduce a new type of support for Borel probability measures.
Let $\mu \in P(\R)$ be an arbitrary measure, and let us define the following sets
$$
N_\mu = \cup\{I\in\irI\colon \mu(I) = 0\}, \quad C_\mu = \R\setminus N_\mu.
$$
The set $C_\mu$ will be called the co-interval support of $\mu$.
Clearly, this is the unique smallest set such that its complement is a union of non-degenerate intervals and $\mu$ is concentrated on it.
This is an analogue of the usual closed support $S_\mu$, however, as we shall see this notion is more useful here.
Recall that $S_\mu$ is the complement of the union of open intervals with zero $\mu$-measure.
Therefore the following properties of the co-interval support are straightforward: 
$$
C_\mu \subseteq S_\mu, \qquad \overline{C_\mu} = S_\mu \quad \text{and} \quad \card(S_\mu \setminus C_\mu) \leq \aleph_0.
$$
If $\irM \subseteq P(\R)$, then let
$$
\irM^1 = \{\nu \in P(\R) \colon d_{Ku}(\mu,\nu) = 1 \;\; \forall \, \mu\in\irM \}.
$$
We call a probability measure which is concentrated on a point $x\in\R$ a Dirac measure, and we will denote it by $\delta_x$.

We begin with the following metric characterisation of Dirac measures.

\begin{lem}\label{Dirac_lem}
For an arbitrary $\mu \in P(\R)$ we have $\card\big((\{\mu\}^1)^1\big) = 1$ if and only if $\mu$ is a Dirac measure.
\end{lem}

\begin{proof}
Let $\nu\in P(\R)$ be an arbitrary measure which is absolutely continuous with respect to $\mu$. 
We show that $\nu \in (\{\mu\}^1)^1$.
By Lemma \ref{Kuiper-c-max_lem}, for any $\vartheta \in \{\mu\}^1$ there exists an $I\in\irI_0$ such that either we have $\mu(I) = 0$ and $\vartheta(I) = 1$, or $\mu(\R\setminus I) = 0$ and $\vartheta(\R\setminus I) = 1$.
Since $\nu$ is absolutely continuous with respect to $\mu$, we immediately infer either $\nu(I) = 0$ and $\vartheta(I) = 1$, or $\nu(\R\setminus I) = 0$ and $\vartheta(\R\setminus I) = 1$, which implies $d_{Ku}(\vartheta,\nu) = 1$.
Since this holds for every $\vartheta \in \{\mu\}^1$, the relation $\nu \in (\{\mu\}^1)^1$ follows.

Now, if $\mu$ is not a Dirac measure, then it is straigthforward that there are infinitely many $\nu \in P(\R)$ which is abolutely contionus with respect to $\mu$.
Hence, we obtain that $\card\big((\{\mu\}^1)^1\big) = 1$ implies that $\mu$ is a Dirac measure.

Finally, we show that $(\{\delta_x\}^1)^1 = \{\delta_x\}$.
Let $\delta_x \in P(\R)$ be a Dirac measure with $x\in\R$.
By Lemma \ref{Kuiper-c-max_lem}, we have $\vartheta \in \{\delta_x\}^1$ if and only if there exists a possibly degenerate interval $I\in\irI_0$ such that either $\delta_x(I) = 0$ and $\vartheta(I) = 1$, or $\delta_x(\R\setminus I) = 0$ and $\vartheta(\R\setminus I) = 1$.
Obviously, in the first case we have $x\in \R\setminus I$, and in the second one $x\in I$.
Therefore we infer the equation
$$
\{\delta_x\}^1 = \{\vartheta\in P(\R) \colon \vartheta(\{x\}) = 0\}.
$$
Assume that $\nu \in (\{\delta_x\}^1)^1$, i.e. $d_{Ku}(\nu,\vartheta) = 1$ for every $\vartheta\in P(\R)$, $\vartheta(\{x\}) = 0$.
Let us suppose for a moment that $\nu(\R\setminus\{x\}) > 0$, and let $\vartheta$ be defined by 
$$
\vartheta(B) = \tfrac{1}{\nu(\R\setminus\{x\})} \cdot \nu(B\setminus\{x\}) \quad (B\in\irB_\R).
$$
It is clear that $\vartheta \in \{\delta_x\}^1$ and that $\vartheta$ is absolutely continuous with respect to $\nu$.
Thus $\nu(I) = 0$ and $\vartheta(I) = 1$ cannot be satisfied simultaneously when $I\in\irI_0$, and the same holds for $\nu(\R\setminus I) = 0$ and $\vartheta(\R\setminus I) = 1$.
This implies $d_{Ku}(\nu,\vartheta) < 1$, a contradiction.
Hence $\nu(\R\setminus\{x\}) = 0$ follows, and we obtain $(\{\delta_x\}^1)^1 = \{\delta_x\}$, which completes the proof.
\end{proof}

We proceed with verifying the following property of surjective Kuiper isometries.

\begin{lem}\label{Dirac-move_lem}
Let $\phi\colon P(\R) \to P(\R)$ be a surjective isometry with respect to the Kuiper distance.
Then there exists a bijection $f \colon \R \to \R$ such that
$$
\phi(\mu)(\{f(x)\}) = \mu(\{x\}) \quad (\mu\in P(\R),\; x\in\R).
$$
\end{lem}

\begin{proof}
Let $\Delta$ denote the set of all Dirac measures of $P(\R)$.
Since $\phi\big((\{\mu\}^1)^1\big) = \big(\phi(\{\mu\}^1)\big)^1 = \big(\{\phi(\mu)\}^1\big)^1$ holds, by Lemma \ref{Dirac_lem} we obtain $\phi(\Delta) = \Delta$, i.e. we have a bijection $f\colon\R\to\R$ such that 
$$
\phi(\delta_x) = \delta_{f(x)} \qquad (x\in\R).
$$

Let $\mu\in P(\R)$ be an arbitrary measure.
We state that 
\begin{equation}\label{Kudist-dx_eq}
d_{Ku}(\mu,\delta_x) = 1-\mu(\{x\}) \qquad (x\in\R).
\end{equation}
On one hand, since $\{x\} \in \irI_0$ and $|\delta_x(\{x\})-\mu(\{x\})| = 1-\mu(\{x\})$, we obtain $d_{Ku}(\mu,\delta_x) \geq 1-\mu(\{x\})$.
On the other hand, let $I\in\irI_0$ be an arbitrary, possibly degenerate, interval.
If $x\in I$, then we have $|\delta_x(I)-\mu(I)| = 1-\mu(I) \leq 1-\mu(\{x\})$; and if $x\notin I$, then we have $|\delta_x(I)-\mu(I)| = \mu(I) \leq \mu(\R\setminus\{x\}) = 1-\mu(\{x\})$.
Therefore we conclude \eqref{Kudist-dx_eq}.

Finally, observing the following for every $x\in\R$ and $\mu\in P(\R)$ completes the proof:
$$
\mu(\{x\}) = 1 - d_{Ku}(\mu,\delta_x) = 1 - d_{Ku}(\phi(\mu),\phi(\delta_x)) = 1 - d_{Ku}(\phi(\mu),\delta_{f(x)}) = \phi(\mu)(\{f(x)\}).
$$
\end{proof}

It is straightforward that for every $\mu\in P(\R)\setminus\Delta$ we can write $N_\mu$ as a disjoint union $N_\mu = (\R\setminus\conv(C_\mu))\cup(\cup_{j=1}^n I_j)$ where $n\in\N\cup\{0,\infty\}$, $I_j\in\irI$ for every $j$, and $\conv(\cdot)$ denotes the convex hull of a given set.
In fact, the $I_j$'s are the bounded connected components of $N_\mu$, and in case when there is at least one unbounded connected component, then the union of these components is exactly $\R\setminus\conv(C_\mu)$.
We note that the intersection of the closure of two connected components can be empty or one point.

Now, we are ready to prove an analogue of Lemma \ref{mu-c-char_lem}, where we describe $\{\mu\}^1\setminus\Delta$ instead of $\{\mu\}^1$.
The reason for this is that the presence of Dirac measures would cause some technical problems.

\begin{lem}\label{mu-1-char_lem}
Let $\mu\in P(\R)\setminus\Delta$ and let us write $N_\mu = (\R\setminus\conv(C_\mu))\cup(\cup_{j=1}^n I_j)$ where the right-hand side is the above mentioned disjoint union.
Then we have
\begin{equation}\label{mu-1_eq}
\{\mu\}^1\setminus\Delta = \big\{ \nu\in P(\R)\setminus\Delta \colon \nu(\R\setminus\conv(C_\mu)) = 1 \big\} \bigcup \left( \bigcup_{j=1}^n \{ \nu\in P(\R)\setminus\Delta \colon \nu(I_j) = 1 \} \right).
\end{equation}
\end{lem}

\begin{proof}
By Lemma \ref{Kuiper-c-max_lem} we have
$$
\{\mu\}^1\setminus\Delta = \{ \nu \in P(\R)\setminus\Delta \colon \; \exists \; I\in\irI_0 \; \nu(I) = 0 \text{ and } \mu(I) = 1\}
$$
$$
\cup \{ \nu \in P(\R)\setminus\Delta \colon \; \exists \; I\in\irI_0 \; \nu(I) = 1 \text{ and } \mu(I) = 0\}
$$
$$
= \{ \nu \in P(\R)\setminus\Delta \colon \; \exists \; I\in\irI \; \nu(I) = 0 \text{ and } \mu(I) = 1\} 
$$
$$
\cup \{ \nu \in P(\R)\setminus\Delta \colon \; \exists \; I\in\irI \; \nu(I) = 1 \text{ and } \mu(I) = 0\}.
$$
On one hand, if we have $\nu(I) = 0$ and $\mu(I) = 1$ with some $I\in\irI$, then we automatically infer $\conv(C_\mu) \subseteq I$.
Therefore this case is equivalent to $\nu(\R\setminus\conv(C_\mu)) = 1$.
On the other hand, if we have $\mu(I) = 0$ and $\nu(I) = 1$ with some $I\in\irI$, then $I\subseteq N_\mu$ is yielded, or equivalently $I\subseteq I_j$ with some $j$ or $I\subseteq \R\setminus\conv(C_\mu)$.
Therefore this case is equivalent to either $\nu(I_j) = 1$ for some $j$, or that the $\nu$-measure of one of the unbounded components of $N_\mu$ is 1.
Noting that this latter possibility implies $\nu(\R\setminus\conv(C_\mu)) = 1$ ends our proof.
\end{proof}

We continue with the following lemma.

\begin{lem}\label{f-hom_lem}
The function $f$ defined in Lemma \ref{Dirac-move_lem} is a homeomorphism (i.e. a monoton bijection).
\end{lem}

\begin{proof}
First, we show the following property:
\begin{equation}\label{co-int-supp-ppres_eq}
C_\mu \subseteq C_\nu \;\iff\; C_{\phi(\mu)} \subseteq C_{\phi(\nu)} \quad (\mu,\nu\in P(\R)).
\end{equation}
If $\mu,\nu \notin \Delta$, then by Lemma \ref{mu-1-char_lem} the following equivalence is yielded:
$$
C_\mu \subseteq C_\nu \;\iff\; \big(\{\nu\}^1\big)^1\setminus\Delta \subseteq \big(\{\mu\}^1\big)^1\setminus\Delta \quad (\mu,\nu\in P(\R)\setminus\Delta).
$$
Since we have $\phi(P(\R)\setminus\Delta) = P(\R)\setminus\Delta$, we get $\phi(\mu),\phi(\nu) \notin \Delta$.
Obviously, we have $\phi\left(\big(\{\mu\}^1\big)^1\setminus\Delta\right) = \phi\left(\big(\{\mu\}^1\big)^1\right)\setminus\Delta = \big(\{\phi(\mu)\}^1\big)^1\setminus\Delta$, whence we conclude \eqref{co-int-supp-ppres_eq} in this case.
Next, we suppose that $\nu = \delta_x$ holds with some $x\in\R$.
Clearly, $C_\mu \subseteq C_{\delta_x}$ is equivalent to $\mu = \delta_x$, which holds if and only if $\phi(\mu) = \phi(\delta_x) = \delta_{f(x)}$.
This latter equation is valid exactly when we have $C_{\phi(\mu)} \subseteq C_{\phi(\delta_x)}$, which gives us \eqref{co-int-supp-ppres_eq} in this case.
Finally, let us assume that we have $\nu\notin\Delta$, $\mu = \delta_x$ and $C_{\delta_x} \subseteq C_\nu$ with some $x\in\R$. 
We consider a point $y\in C_\nu$, $x\neq y$.
Since we have $C_{\tfrac{1}{2}\delta_x+\tfrac{1}{2}\delta_y} \subseteq C_\nu$, by the previous cases, we obtain $C_{\phi\left(\tfrac{1}{2}\delta_x+\tfrac{1}{2}\delta_y\right)} \subseteq C_{\phi(\nu)}$.
Moreover, by Lemma \ref{Dirac-move_lem}, we have $\phi\left(\tfrac{1}{2}\delta_x+\tfrac{1}{2}\delta_y\right) = \tfrac{1}{2}\phi(\delta_x)+\tfrac{1}{2}\phi(\delta_y) = \tfrac{1}{2}\delta_{f(x)}+\tfrac{1}{2}\delta_{f(y)}$, thus we obtain $C_{\phi(\delta_x)} = C_{\delta_{f(x)}} \subseteq C_{\phi\left(\tfrac{1}{2}\delta_x+\tfrac{1}{2}\delta_y\right)} \subseteq C_{\phi(\nu)}$.
This verifies \eqref{co-int-supp-ppres_eq} in this case in one direction.
The other direction follows from the fact that $\phi^{-1}$ is also a surjective Kuiper isometry.
Therefore we conclude \eqref{co-int-supp-ppres_eq} in general.

Next, by Lemma \ref{Dirac-move_lem}, we have $\phi(P_c(\R)) = P_c(\R)$.
Theorem \ref{main-c_thm} gives us a homeomorphism $g\colon \R\to\R$ and an $x\in\R\cup\{\infty\}$ such that we have
$$
\phi(\mu) = \mu\circ (g \circ r_x) \qquad (\mu \in P_c(\R)),
$$
and consequently we obtain
$$
S_{\phi(\mu)} = \left\{
\begin{matrix}
g^{-1}(S_\mu) & \text{if } x=\infty\\
r_0 \circ g^{-1}(S_\mu) + x  & \text{if } x\in\R \text{ and } S_\mu \text{ is compact} \\
\{x\}\cup(r_0 \circ g^{-1}(S_\mu) + x)  & \text{if } x\in\R \text{ and } S_\mu \text{ is not compact} 
\end{matrix}
\right. \quad (\mu \in P_c(\R)).
$$
Let us suppose for a moment that we have $x\in\R$, and consider a sequence of continuous measures $\{\mu_n\}_{n=1}^\infty$ with 
$$
C_{\mu_n} = 
\left\{ 
\begin{matrix} 
\left(g\left(-\tfrac{1}{n}\right), g\left(\tfrac{1}{n}\right)\right) & \text{if } g \text{ is monotone increasing}\\
\left(g\left(\tfrac{1}{n}\right), g\left(-\tfrac{1}{n}\right)\right) & \text{if } g \text{ is monotone decreasing}
\end{matrix}
\right..
$$
Since $\{g(0)\} = C_{\delta_{g(0)}} \subseteq C_{\mu_n}$, we infer $\{f(g(0))\} = C_{\phi(\delta_{g(0)})} \subseteq C_{\phi(\mu_n)} = \R\setminus [-n+x,n+x]$ for every $n\in\N$, which gives the contradiction $\{f(g(0))\} \subseteq \cap_{n=1}^\infty \left(\R\setminus [-n+x,n+x]\right) = \emptyset$.
Therefore we conclude that 
$$
\phi(\mu) = \mu\circ g \; \text{and} \; S_{\phi(\mu)} = g^{-1}(S_\mu) \quad (\mu \in P_c(\R)),
$$
holds with some homeomorphism $g$.

Now, we consider a number $t\in\R$ and a sequence $\{\nu_n\}_{n=1}^\infty \subset P_c(\R)$ such that
$$
C_{\nu_n} = 
\left\{ 
\begin{matrix} 
\left(g\left(g^{-1}(t)-\tfrac{1}{n}\right), g\left(g^{-1}(t)+\tfrac{1}{n}\right)\right) & \text{if } g \text{ is monotone increasing}\\
\left(g\left(g^{-1}(t)+\tfrac{1}{n}\right), g\left(g^{-1}(t)-\tfrac{1}{n}\right)\right) & \text{if } g \text{ is monotone decreasing}
\end{matrix}
\right..
$$
Since we have $\{t\} = C_{\delta_{t}} \subseteq C_{\nu_n}$ $(n\in\N)$, we obtain $\{f(t)\} = C_{\phi(\delta_{t})} \subseteq \cap_{n=1}^\infty C_{\phi(\nu_n)} = \{g^{-1}(t)\}$, which implies $g = f^{-1}$.
As a consequence we have that $f$ is a homeomorphism, which makes the proof complete.
\end{proof}

Before we prove Theorem \ref{main-gen_thm}, we need one further statement.

\begin{lem}\label{d-dense_lem}
The set of all purely atomic probability measures are dense in $P(\R)$ with respect to the Kuiper distance.
\end{lem}

\begin{proof}
Here it is easier to consider distribution functions instead of measures.
Let us observe that on the space of all distribution functions the Kuiper distance is equivalent to the Kolmogorov--Smirnov distance, which is obtained from the supremum norm.
Since the statement is quite straightforward to show in the Kolmogorov--Smirnov distance, we also have it for the Kuiper distance.
\end{proof}

Finally, we are in the position to present the verification of our last result.

\begin{proof}[Proof of Theorem \ref{main-gen_thm}]
Assume that $\phi\colon P(\R) \to P(\R)$ is a surjective Kuiper isometry.
Let us consider the function $f \colon \R \to \R$ which was defined in Lemma \ref{Dirac-move_lem} and which is a homeomorphism by Lemma \ref{f-hom_lem}.
Therefore 
$$
\phi_1\colon P(\R) \to P(\R), \; \phi_1(\mu) = \phi(\mu)\circ f \; (\mu\in P(\R))
$$
is also a surjective isometry with respect to the Kuiper distance, and we have
$$
(\phi_1(\mu))(\{x\}) = \mu(\{x\}) \quad (x\in\R, \, \mu\in P(\R)).
$$
Consequently, we obtain 
$$
\phi_1\left( \sum_{j=1}^n \alpha_j \delta_{x_j} \right) = \sum_{j=1}^n \alpha_j \delta_{x_j} \quad \left(j\in\N\cup\{\infty\},\, x_j\in\R,\, \alpha_j > 0, \; \sum_{j=1}^n \alpha_j = 1\right).
$$
Since isometries are automatically continuous, Lemma \ref{d-dense_lem} implies that $\phi_1$ has to be the identity map.
Therefore, we conclude $\phi(\mu) = \mu\circ f^{-1}$ $(\mu\in P(\R))$, which completes the proof.
\end{proof}


\section*{Acknowledgments}
The author was also supported by the "Lend\" ulet" Program (LP2012-46/2012) of the Hungarian Academy of Sciences and by the Hungarian National Research, Development and Innovation Office -- NKFIH (grant no.~K115383).


\bibliographystyle{amsplain}

\end{document}